\documentclass[12pt]{article}

\usepackage{amsmath,amsthm,amsfonts,amssymb,color}

\newtheorem{theorem}{Theorem}[section]

\newtheorem{proposition}[theorem]{Proposition}
\newtheorem{lemma}[theorem]{Lemma}
\newtheorem{definition}[theorem]{Definition}
\newtheorem{remark}[theorem]{Remark}

\def\cD{\mathcal{D}}
\def\cF{\mathcal{F}}
\def\cG{\mathcal{G}}
\def\cH{\mathcal{H}}
\def\cS{\mathcal{S}}

\def\bR{\mathbb{R}}

\def\bD{\mathbb{D}}

\def\conda{|\mathcal F_A}
\def\fa{\mathcal F_A}
\def\ia{\mathbf{1}_{A}}
\def\iac{\mathbf{1}_{A^c}}

\newcommand*\samethanks[1][\value{footnote}]{\footnotemark[#1]}
\begin{document}

\title{Intermittency for the Hyperbolic Anderson Model  \\
with rough noise in space}

\author{Raluca Balan\footnote{Department of Mathematics and Statistics, University of Ottawa,
585 King Edward Avenue, Ottawa, ON, K1N 6N5, Canada. E-mail address:
rbalan@uottawa.ca. Research supported by a grant from the Natural
Sciences and Engineering Research Council of Canada.} \and Maria
Jolis\thanks{Departament de Matem\`atiques, Universitat Aut\`{o}noma
de Barcelona, 08193 Bellaterra (Barcelona), Catalonia, Spain. E-mail
addresses: mjolis@mat.uab.cat, quer@mat.uab.cat. Research supported
by the grant MCI-FEDER MTM2012-33937.} \and Llu\'{i}s
Quer-Sardanyons \samethanks[2] \footnote{Corresponding author.}
 }

\date{April 20, 2016}

\maketitle

\begin{abstract}
\noindent In this article, we consider the stochastic wave equation
on the real line driven by a linear multiplicative Gaussian noise,
which is white in time and whose spatial correlation corresponds to
that of a fractional Brownian motion with Hurst index $H\in
(\frac14,\frac12)$. Initial data are assumed to be constant. First,
we prove that this equation has a unique solution (in the Skorohod
sense) and obtain an exponential upper bound for the $p$-th moment
of the solution, for any $p\geq 2$. Condition $H>\frac14$ turns out
to be necessary for the existence of solution. Secondly, we show
that this solution coincides with the one obtained by the authors in
a recent publication, in which the solution is interpreted in the
It\^o sense. Finally, we prove that the solution of the equation in
the Skorohod sense is weakly intermittent.
\end{abstract}

\bigskip
\bigskip
\bigskip


\bigskip



\newpage

\section{Introduction}

In this article, we continue the investigations from the recent article \cite{BJQ},
by focusing on the stochastic wave equation with constant initial conditions and a linear
term $\sigma(u)=\lambda u$ multiplying the noise:
\begin{equation}
\left\{\begin{array}{rcl} \displaystyle \frac{\partial^2 u}{\partial
t^2}(t,x) & = & \displaystyle \frac{\partial^2 u}{\partial x^2}(t,x)
+\lambda u(t,x)\dot{X}(t,x), \quad t>0, \, x \in \bR \\[2ex]
\displaystyle u(0,x) & = & \eta, \quad x \in \bR\\[1ex]
\displaystyle \frac{\partial u}{\partial t}(0,x) & = & 0, \quad x
\in \bR
\end{array}\right. \label{HAM}
\end{equation}
where $\lambda \in \bR$ and $\eta \in \bR$. To avoid trivial
situations, we assume that $\lambda\not=0$ and $\eta \not=0$.
This problem is known in the literature as the Hyperbolic Anderson Model, by analogy with its parabolic counterpart.

As in reference \cite{BJQ}, we assume that the noise $\dot{X}$ is white in
time and behaves in space like the formal derivative of a
fractional Brownian motion with Hurst index $H\in (\frac14,\frac
12)$. This noise is given by a zero-mean Gaussian process $X=\{X(\varphi);\varphi \in
\cD({\mathbb{R}_+}\times \bR)\}$, defined on a complete probability space
$(\Omega,\mathfrak{F},P)$, with covariance:
$$E[X(\varphi)X(\psi)]=\int_{0}^{\infty}\int_{\bR}\cF \varphi(t,\cdot)(\xi)\overline{\cF \psi(t,\cdot)
(\xi)}\mu(d\xi)dt, \quad \varphi, \psi\in \cD({\mathbb{R}_+}\times
\bR),$$
where $\cD({\mathbb{R}_+}\times
\bR)$ is the space of {infinitely differentiable functions on $\bR_{+} \times \bR$ with
compact support, and $\cF \varphi(t,\cdot)$ is the Fourier transform of the function $\varphi(t,\cdot)$, defined by:
$$\cF \varphi(t,\cdot)(\xi)=\int_{\bR} e^{-i \xi x}\varphi(t,x)dx, \quad \xi \in \bR.$$

We assume that the measure $\mu$ is given by $\mu(d\xi)=c_H
|\xi|^{1-2H}d \xi$, with $\frac14<H<\frac12$ and
$$c_H=\frac{\Gamma(2H+1)\sin(\pi H)}{2\pi}.$$

Since the Fourier transform of $\mu$ in the space $\cS'(\bR)$ of tempered distributions on $\bR$ is {\em not} a
locally integrable function, the noise $X$ is not of the same form as the one considered by Robert Dalang
in his seminal article \cite{Dalang}, and the stochastic integral with respect to $X$ was constructed in \cite{BJQ}
using different methods.

\smallskip

In the present article, the solution of equation \eqref{HAM} is
defined using the divergence operator from Malliavin calculus, as
opposed to the It\^o-type stochastic integral used in article
\cite{BJQ}. We say that the two solutions are interpreted in the
Skorohod sense, respectively the It\^o sense. Nevertheless, we will
show that the two solutions coincide, by extending to the case of
the noise $X$ a classical result from Malliavin calculus which says
that the Skorohod integral of a {\em measurable and adapted} process
with respect to the Brownian motion coincides with its It\^o
integral (see Section \ref{sec:ito} for details.)

\smallskip

In the first part of the paper, we show that equation
\eqref{HAM} has a unique solution (in the Skorohod sense),
whose moments of order
$p\geq 2$ are bounded by an exponential
function of $t$, up to some constants; see Theorem
\ref{thm:existence} for the precise statement.
The proof of this result is based on Malliavin calculus techniques. More precisely, we write the Wiener chaos expansion of the solution $u(t,x)$ in the Skorohod sense, and we
estimate the second moment of each multiple Wiener integral which appear in this expansion, following very closely the method used in \cite{HHLNT} for the parabolic case.

\smallskip

This methodology has been used in the case of the stochastic heat
equation in \cite{hu01,hu-nualart09,HNS11,BT10}, and also in the
recent paper \cite{HHLNT}, in which the noise is the same as the one
considered in the present article. In fact, article \cite{HHLNT} contains a thorough analysis of the heat equation
with a general diffusion coefficient $\sigma(u)$ multiplying the noise.
On the other hand, the stochastic wave equation driven by a noise different than the one considered here
(either smoother in space or fractional in time) was studied in references \cite{DMT08,DM09,B12,balan-conus}.
Finally, the
stochastic wave and heat equations with an
affine diffusion coefficient $\sigma(u)=au+b$ and the same noise $X$ as here have been studied in
the recent article \cite{BJQ}, using the classical method of Picard iterations, the solution being interpreted in the It\^o sense. As explained above, it turns out that the solution to equation
\eqref{HAM} in the Skorohod sense coincides with the one obtained in
\cite{BJQ}. The advantage of the method based on the Wiener chaos expansion is that it allows us obtain estimates
for the $p$-th moments of the solution, which lead to the weak intermittency property of the solution.

\smallskip

Recently, there has been a lot of interest in
studying the intermittency property of solutions to stochastic partial differential equations, such as the
stochastic heat and wave equations. For the former, we refer the
reader to \cite{BC95,CK12,FK13,CD15,HHNT}. On the other hand,
intermittency for the solution of the stochastic wave equation driven by a noise
which is white in time and has a smoother space correlation than
the one considered here was studied in \cite{DM09,C13}.
Finally, in
\cite{balan-conus}, it was proved that the solutions to the stochastic wave and heat equations exhibit an
intermittency-type property, even when the noise is fractional in time and has the same spatial correlation
as in Dalang's article \cite{Dalang}.

\smallskip

The notion of intermittency which will be considered here is {\it weak intermittency}, which is
defined as follows. Recall that the {\it lower} and {\it upper
Lyapunov exponents} of order $p \geq 2$ of the random field $u=\{u(t,x);
t\geq 0, x\in \bR\}$ which solves equation \eqref{HAM} are defined, respectively, by:
\begin{equation}
 \underline{\gamma}(p):=\liminf_{t \to \infty}\frac{1}{t}\,\inf_{x\in\bR}\,\log E|u(t,x)|^p
 \label{def-Lyapunov}
\end{equation}
and
\begin{equation}
\overline{\gamma}(p):=\limsup_{t \to \infty}\frac{1}{t}\,\sup_{x\in\bR}\,\log E|u(t,x)|^p.
\label{def-Lyapunov-up}
\end{equation}
The goal of the present article is to
prove that $u$ is {\it weakly intermittent}, which means that
$$\underline{\gamma}(2)>0 \quad \mbox{and} \quad
\overline{\gamma}(p)<\infty \quad \mbox{for all} \ p \geq 2.$$
The physical interpretation of this property is
that, as time becomes large, the paths of the random field $u$
exhibit very high peaks concentrated on small spatial islands.

\smallskip

We observe that, in most of the references cited above, a weaker notion
of {\it weak intermittence} has been considered, namely $\overline{\gamma}(2)>0$ and
$\overline{\gamma}(p)<+\infty$ for all $p\geq 2$.

\smallskip

As noticed in \cite{B12}, the fact that $\overline{\gamma}(p)<\infty$ for
all $p \geq 2$ is a direct consequence of the second moment estimate obtained
for the $n$-th term appearing in the Wiener chaos expansion of $u(t,x)$. Condition
$\underline{\gamma}(2)>0$ turns out to be more delicate and is proved by
showing that the second moment of the solution admits a lower exponential estimate.
Though our noise is
rougher than the one considered in some related references, e.g. \cite{FK13}, where
the above-mentioned weaker notion of {\it weak intermittency} is considered,
we succeeded to write a rather simplified proof of $\underline{\gamma}(2)>0$
by taking advantage of the
noise's roughness. More precisely, the fact
that $H<\frac12$ implies that the function $\xi \to \xi^{1-2H}$ is
increasing on $\bR_+$, which allowed us to find a suitable
lower bound for $\underline{\gamma}(2)$; see the proof of Theorem
\ref{thm:main} for details.

\smallskip

We notice that, though intermittency properties for the
parabolic Anderson model driven by the noise $X$ have already
been studied in \cite{HHLNT}, the proof of our Theorem \ref{thm:main} also works in this latter case;
see Remark \ref{rmk:01}.

\smallskip

Finally, we should mention that, as in \cite{BJQ} and \cite{HHLNT}, we were not able to
eliminate of the restriction $H>\frac14$, which is needed for the existence of the solution.
This restriction might have a deeper meaning, but from the technical point of view, it comes from the requirement
\[
 \int_\bR |\cF G(t,\cdot)(\xi)|^2 |\xi|^{2(1-2H)} d\xi<\infty,
\]
which arises naturally in our calculations.
Here
$G(t,x)=\frac12 {\bf 1}_{\{|x|<t\}}$, $t>0$ and $x\in \bR$, is
the fundamental solution of the wave equation in $\bR$.
Indeed, we prove that $H>\frac14$ is a necessary condition for the existence of the
solution to equation \eqref{HAM} with noise $X$ with $0<H<\frac12$ (see Proposition \ref{prop:optimality}).

\smallskip

The paper is organized as follows. In Section \ref{sec:Malliavin}, we
present the Malliavin calculus setting associated to our noise and
introduce the tools needed in the sequel. In Section \ref{sec:solution},
we prove that equation
\eqref{HAM} has a unique solution in the Skorohod sense. In Section \ref{sec:ito}, we
compare the solution obtained in the previous section with the
It\^o-type solution studied in \cite{BJQ}. An important step here
is to compare the Skorohod-type integral with respect to our
noise with the It\^o-type integral defined in the latter reference.
This result is of independent interest and is proved
in Appendix \ref{appendix}. Finally, in Section
\ref{sec:inter}, we prove that the solution to equation \eqref{HAM} is weakly intermittent,
in the sense described above.

\smallskip

Along the paper we use the notation $C$ for any positive real constant, independently of its value.


\section{Preliminaries on Malliavin calculus}
\label{sec:Malliavin}

In order to give a meaning of solution to equation \eqref{HAM}, we
use an approach based on Malliavin calculus with respect to the
isonormal Gaussian process determined by the noise $\dot{X}$. We
describe briefly this procedure. We refer to \cite{nualart06} for
more details.

Let $\cH$ be the completion of $\cD({\mathbb{R}_+}\times \bR)$ with
respect to $\langle \cdot,\cdot \rangle_{\cH}$, where
\[
\langle \varphi,\psi \rangle_{\cH}:= \int_{0}^{\infty}\int_{\bR}\cF
\varphi(t,\cdot)(\xi)\overline{\cF \psi(t,\cdot) (\xi)}\mu(d\xi)dt.
\]
The map $\varphi \mapsto X(\varphi) \in L^2(\Omega)$ is an isometry
which can be extended to $\cH$. We denote this map by
$$X(\varphi)=\int_{0}^{\infty}\int_{\bR}\varphi(t,x)X(dt,dx), \quad \varphi \in \cH.$$
We say that $X(\varphi)$ is the Wiener integral of $\varphi$ with
respect to $X$. Then, $\{X(\varphi);\varphi \in \cH\}$ defines an
isonormal Gaussian process and we can develop the Malliavin calculus
techniques based on it.

Recall that the square of the norm in $\cH$ can also be written as
follows:
$$
\|\varphi\|_{\cH}^2= C_H \int_0^{\infty} \int_{\bR}
\int_{\bR}|\varphi(x)-\varphi(y)|^2|x-y|^{2-2H} \,dxdydt,$$ where
$C_H=H(1-2H)/2$.

Moreover, it can be seen that $\cH$ is a space of functions in both variables $t$ and $x$; see
\cite{Jo} for the proof of this fact, in the case when the noise is independent of $t$.

Let $\cG$ be the $\sigma$-field generated by $\{X(\varphi);\varphi
\in \cH\}$. By Theorem 1.1.1 of \cite{nualart06}, every random
variable $F \in L^2(\Omega,\cG,P)$ has the Wiener chaos expansion:
$$F=E(F)+\sum_{n \geq 1}F_n \quad \mbox{with} \quad F_n \in \cH_n,$$
where $\cH_n$ is the $n$-th Wiener chaos space associated to $X$.

Each random variable $F \in \cH_n$ can be represented as $F=I_n(f)$
for some $f \in \cH^{\otimes n}$, where $\cH^{\otimes n}$ is the
$n$-th tensor product of $\cH$ and $I_n:\cH^{\otimes n} \to \cH_n$
is the multiple Wiener integral with respect to $X$ (see, e.g.
\cite[Prop. 1.1.4]{nualart06}). In our case, the norm in
$\cH^{\otimes n}$ is given by:
\begin{equation}
\label{def-norm-Hn} \|f\|_{\cH^{\otimes n}}^2 =
\int_{\bR_{+}^n}\int_{\bR^n}|\cF
f(t_1,\cdot,\ldots,t_n,\cdot)(\xi_1, \ldots,\xi_n)|^2 \mu(d\xi_1)
\ldots \mu(d\xi_n) dt_1 \ldots dt_n.
\end{equation}

For any $f \in \cH^{\otimes n}$,
\begin{equation}
\label{int-sym-equal} I_n(f)=I_n(\widetilde{f})
\end{equation} and
$$E|I_n(f)|^2=E|I_n(\widetilde{f})|^2=n! \, \|\widetilde{f}\|_{\cH^{\otimes n}}^2,$$
where $\widetilde{f}$ is the symmetrization of $f$ in all $n$
variables:
$$\widetilde{f}(t_1,x_1,\ldots,t_n,x_n)=\frac{1}{n!}\sum_{\rho \in S_n}f(t_{\rho(1)},x_{\rho(1)},\ldots,t_{\rho(n)},x_{\rho(n)}),$$
and we denote by $S_n$ the set of all permutations of $\{1,
\ldots,n\}$. By the orthogonality of the Wiener chaos spaces, for
any $f \in \cH_n$ and $g \in \cH_m$,
\begin{equation}
\label{orthog-Hn} E[I_n(f)I_m(g)]=\left\{
\begin{array}{ll} n! \, \langle \widetilde{f}, \widetilde{g} \rangle_{\cH^{\otimes n}} & \mbox{if $n=m$} \\
0 & \mbox{if $n \not=m$}
\end{array} \right.
\end{equation}
The Wiener chaos expansion of $F \in L^2(\Omega,\cG,P)$ becomes:
$$F=\sum_{n \geq 0}I_n(f_n)=\sum_{n \geq 0}I_n(\widetilde{f}_n),$$
where $f_n \in \cH^{\otimes n}$ for $n\geq 1$,
$f_0=\widetilde{f}_0=E(F)$ and $I_0:\bR \to \bR$ is the identity
map. Using again the orthogonality of the Wiener chaos spaces, we
obtain that
$$E|F|^2=\sum_{n \geq 0}E|I_n(f_n)|^2=\sum_{n \geq 0}n! \, \|\widetilde{f}_n\|_{\cH^{\otimes n}}^{2}.$$

Let $\cS$ be the class of smooth random variables of the form
\begin{equation}
\label{form-F}F=f(X(\varphi_1),\ldots, X(\varphi_n)),
\end{equation} where $f \in C_{b}^{\infty}(\bR^n)$, $\varphi_i \in \cH$, $n \geq 1$, and
$C_b^{\infty}(\bR^n)$ is the class of bounded $C^{\infty}$-functions
on $\bR^n$, whose partial derivatives of all orders are bounded. The
{\bf Malliavin derivative} of $F$ of the form (\ref{form-F}) is an
$\cH$-valued random variable given by:
$$DF:=\sum_{i=1}^{n}\frac{\partial f}{\partial x_i}(X(\varphi_1),\ldots,
X(\varphi_n))\varphi_i.$$ We endow $\cS$ with the norm
$\|F\|_{\bD^{1,2}}:=(E|F|^2)^{1/2}+(E\|D F \|_{\cH}^{2})^{1/2}$. The
operator $D$ can be extended to the space $\bD^{1,2}$, the
completion of $\cS$ with respect to $\|\cdot \|_{\bD^{1,2}}$.

The {\bf divergence operator} $\delta$ is defined as the adjoint of
the operator $D$. The domain of $\delta$, denoted by $\mbox{Dom} \
\delta$, is the set of $u \in L^2(\Omega;\cH)$ such that
$$|E \langle DF,u \rangle_{\cH}| \leq c (E|F|^2)^{1/2}, \quad \forall F \in \bD^{1,2},$$
where $c$ is a constant depending on $u$. If $u \in {\rm Dom} \
\delta$, then $\delta(u)$ is the element of $L^2(\Omega)$
characterized by the following duality relation:
\begin{equation}
\label{duality} E(F \delta(u))=E\langle DF,u \rangle_{\cH}, \quad
\forall F \in \bD^{1,2}.
 \end{equation}
In particular, $E[\delta(u)]=0$. If $u \in \mbox{Dom} \ \delta$, we
use the notation
$$\delta(u)=\int_0^{\infty} \int_{\bR^d}u(t,x) X(\delta t, \delta x),$$
and we say that $\delta(u)$ is the {\bf Skorohod integral} of $u$
with respect to $X$.

\bigskip

The following result is the analogue of Proposition 1.3.7 of
\cite{nualart06} for the noise $X$.

\begin{proposition}
\label{integr-criterion} Let $u\in L^2(\Omega;\cH)$ such that for
each $t>0$ and $x \in \bR$, $u(t,x) \in L^2(\Omega,\cG,P)$ has the
Wiener chaos expansion
$$
u(t,x)=\sum_{n \geq 0}I_n(f_n(\cdot,t,x)),$$
 for some $f_n(\cdot,t,x)\in \cH^{\otimes n}$. Then $u \in {\rm Dom} \ \delta$ if and only if the series $\sum_{n \geq 1}I_{n+1}(f_n)$ converges in $L^2(\Omega)$. In this case,
$$\delta(u)=\sum_{n \geq 0}I_{n+1}(f_n)=\sum_{n \geq 0}I_{n+1}(\widetilde{f}_n),$$
where $\widetilde{f}_n$ is the symmetrization of $f_n$ in all $n+1$
variables.
\end{proposition}


\section{Solution in the Skorohod sense}
\label{sec:solution}

In this section, we define the concept of solution to equation \eqref{HAM} in the Skorohod sense,
and we prove that this solution exists and is unique. As a by-product of this procedure,
we obtain immediately an exponential upper bound for the $p$-th moment of the solution,
which means that $\overline{\gamma}(p)<\infty$ for any $p \geq 2$.
Moreover, we also prove that condition $H>\frac14$ is necessary for the existence and
uniqueness of solution to equation \eqref{HAM}.

\medskip

Let $G$ be the Green function of the wave operator on $\bR_{+}
\times \bR$, i.e.
$$G(t,x)=\frac{1}{2}1_{\{|x|<t\}}, \quad t>0,x \in \bR.$$
We consider the filtration
$$\cF_t=\sigma (\{X(1_{[0,s]}\varphi);0 \leq s \leq t,\varphi \in \cD(\bR)\}) \vee {\cal N}, \quad t \geq 0,$$
where ${\cal N}=\{F \in \mathfrak{F}; P(F)=0 \}$.

\begin{definition}
\label{definition-solution} {\rm We say that a process $u=\{u(t,x);t
\geq 0, x \in \bR^d\}$ is a {\bf solution of \eqref{HAM} (in the
Skorohod sense)} if, for any $t \geq 0$ and $x \in \bR$,  $E|u(t,x)|^2<\infty$ and
\begin{equation}
\label{def-sol} u(t,x) = \eta + \int_{0}^{t}\int_{\bR}G(t-s,x-y)
\lambda u(s,y)X(\delta s,\delta y),
\end{equation}
i.e. the process $v^{(t,x)}= \{1_{[0,t]}(s)\lambda G(t-s,x-y)
u(s,y); s \geq 0,y \in \bR\}$ belongs to ${\rm Dom} \ \delta$ and
$u(t,x)=\eta+\delta(v^{(t,x)})$. }
\end{definition}

To see when the solution exists, for any $t>0$ and $x \in \bR$, we
define $f_0(t,x)=\eta$, and
\begin{align}
& f_n(t_1,x_1, \ldots,t_n,x_n,t,x) \nonumber \\
& \qquad \qquad =\lambda^n  G(t-t_n,x-x_n) \ldots
G(t_2-t_1,x_2-x_1)\, \eta 1_{\{0<t_1<\ldots<t_n<t\}}, \label{def-fn}
\end{align}
for $n \geq 1$. We let $\widetilde{f}_0(t,x)=\eta$ and for $n\geq
1$, we let $\widetilde{f}_n(\cdot,t,x)$ be the symmetrization of
$f_n(\cdot,t,x)$:
\begin{align}
& \widetilde{f}_n(t_1,x_1,\ldots,t_n,t,x) \nonumber \\
&  \quad =\eta \frac{\lambda^n}{n!}\sum_{\rho \in S_n}
G(t-t_{\rho(n)},x-x_{\rho(n)}) \ldots
G(t_{\rho(2)}-t_{\rho(1)},x_{\rho(2)}-x_{\rho(1)})
1_{\{0<t_{\rho(1)}<\ldots<t_{\rho(n)}<t\}}. \label{def-fn-sym}
\end{align}
The main result of the section (see Theorem \ref{thm:existence}
below) is based on the following proposition.

\begin{proposition}
\label{exist-sol-th} Equation \eqref{HAM} has a solution if and only
if, for any $t>0$ and $x \in \bR$, it holds
\begin{equation}
\label{series-conv} \sum_{n \geq 0}n!\,
\|\widetilde{f}_n(\cdot,t,x)\|_{\cH^{\otimes n}}^{2}<\infty.
\end{equation}
In this case, the solution $u=\{u(t,x);t \geq 0,x \in \bR\}$ is
unique and has the Wiener chaos expansion:
\begin{equation}
\label{chaos-exp} u(t,x)=\sum_{n\geq 0}I_n(f_n(\cdot,t,x)),
\end{equation}
with kernels $f_n(\cdot,t,x), n \geq 1$ given by \eqref{def-fn} and
$f_0(t,x)=\eta$. Moreover, for any $t \geq 0$ and $x \in \bR$,
$$E|u(t,x)|^2=\sum_{n \geq 0}n!\, \|\widetilde{f}_n(\cdot,t,x)\|_{\cH^{\otimes n}}^{2}.$$
\end{proposition}

\begin{proof} We use the same argument as on p. 302-303 of \cite{hu-nualart09} for the Parabolic Anderson Model (with a noise different than here).
Assume that a solution $u=\{u(t,x)\}$ to equation \eqref{HAM}
exists. Since $u(t,x)\in L^2(\Omega,\cG,P)$, it has the Wiener chaos
expansion \eqref{chaos-exp} for some $f_n(\cdot,t,x) \in
\cH^{\otimes n}$ for $n \geq 1$ and $f_0(t,x)=\eta$.

We fix $t>0$ and $x \in \bR$ and we write the Wiener chaos expansion
for the variable $u(s,y)$ for $s \in [0,t]$ and $y \in \bR$. We
multiply this by the deterministic function $1_{[0,t]}(s)\lambda
G(t-s-x-y)$. It follows that the process $v^{(t,x)}$ given in
Definition \ref{definition-solution} has the Wiener chaos expansion
$v^{(t,x)}(s,y)=\sum_{n \geq 0}I_{n}(g_n^{(t,x)}(\cdot,s,y))$, with
kernels:
\begin{equation}
\label{def-gn} g_{n}^{(t,x)}(\cdot,s,y)=1_{[0,t]}(s)\lambda
G(t-s,x-y)\widetilde{f}_n(\cdot,s,y),
\end{equation}
being $\widetilde{f}_n(\cdot,s,y)$ the symmetrization of
$f_n(\cdot,s,y)$ in the first $n$ variables.

 By Proposition \ref{integr-criterion},
$v^{(t,x)} \in {\rm Dom} \ \delta$ if and only if $\sum_{n \geq
0}I_{n+1}(g_n^{(t,x)})$ converges in $L^2(\Omega)$. In this case,
$\delta(v^{(t,x)})=\sum_{n \geq 0}I_{n+1}(g_{n}^{(t,x)})$ and
relation $u(t,x)=\eta+\delta(v^{(t,x)})$ becomes:
\begin{equation}
\label{chaos-eq} \sum_{n \geq 0}I_n(f_n(\cdot,t,x))=\eta+\sum_{n
\geq 0}I_{n+1}(g_n^{(t,x)}).
\end{equation}

We denote  by $\widetilde{g_n^{(t,x)}}$ the symmetrization of $g$ in
all $n+1$ variables, i.e.
\begin{align*}
& \widetilde{g_n^{(t,x)}}(t_1,x_1,\ldots,t_n,x_n,s,y) \\
& \qquad =
\frac{1}{n+1}\Big[
g_n^{(t,x)}(t_1,x_1,\ldots,t_n,x_n,s,y)  \\
& \quad \quad \quad +\sum_{i=1}^{n}
g_n^{(t,x)}(t_1,x_1,\ldots,t_{i-1},x_{i-1},s,y,t_{i+1},x_{i+1},\ldots,t_n,x_n,t_i,x_i)
\Big]
\end{align*}
Relation \eqref{chaos-eq} can be written also as:
$$\eta+\sum_{n\geq 0}I_{n+1}(\widetilde{f}_{n+1}(\cdot,t,x))=\eta+\sum_{n\geq 0}I_{n}(\widetilde{g_n^{(t,x)}}).$$
By the uniqueness of the Wiener chaos expansion with symmetric
kernels (see Theorem 1.1.2 of \cite{nualart06}), we infer that for
any $n \geq 0$,
$$\widetilde{f}_{n+1}(\cdot,t,x)=\widetilde{g_n^{(t,x)}},$$
that is
$$\widetilde{f}_{n+1}(t_1,x_1,\ldots,t_n,x_n,t_{n+1},x_{n+1},t,x)=
\widetilde{g_n^{(t,x)}}(t_1,x_1,\ldots,t_n,x_n,t_{n+1},x_{n+1}).$$
This allows us to find $\widetilde{f}_n(\cdot,t,x)$ recursively:
\begin{eqnarray*}
\widetilde{f}_1(t_1,x_1,t,x)&=&\widetilde{g_0^{(t,x)}}(t_1,x_1)=1_{[0,t]}(t_1)
\lambda G(t-t_1,x-x_1)\eta \\
\widetilde{f}_2(t_1,x_1,t_2,x_2,t,x)&=&\widetilde{g_1^{(t,x)}}(t_1,x_1,t_2,x_2)\\
&=& \frac{1}{2}\left[g_1^{(t,x)}(t_1,x_1,t_2,x_2)+g_1^{(t,x)}(t_2,x_2,t_1,x_1)\right] \\
&=& \frac{1}{2}[1_{[0,t]}(t_2)\lambda
G(t-t_2,x-x_2)f_1(t_1,x_1,t_2,x_2)+\\
& & 1_{[0,t]}(t_1) \lambda
G(t-t_1,x-x_1f_1(t_2,x_2,t_1,x_1) ]\\
&=& \frac{\lambda^2}{2}[1_{[0,t]}(t_2)
G(t-t_2,x-x_2) 1_{[0,t_2]}(t_1)G(t_2-t_1,x_2-x_1)\eta+\\
& & 1_{[0,t]}(t_1) G(t-t_1,x-x_1)
1_{[0,t_1]}(t_2)G(t_1-t_2,x_1-x_2)\eta],
\end{eqnarray*}
and so on. This shows that the kernels $\widetilde{f}_n(\cdot,t,x)$ must be of
the form \eqref{def-fn-sym}.

The series $\sum_{n \geq 0}I_{n+1}(g_n^{(t,x)})=\sum_{n \geq
0}I_{n+1}(f_{n+1}(\cdot,t,x))$ converges in $L^2(\Omega)$ if and
only if \eqref{series-conv} holds. In this case, the solution
$u$ exists and is unique, with the Wiener chaos expansion
\eqref{chaos-exp} with kernels $f_n(\cdot,t,x)$ given by
\eqref{def-fn}.
\end{proof}

\medskip

We will also need the following technical result which follows from
Lemma 4.5 of \cite{HHNT} using the change of variable
$s_j=t-t_{n+1-j}$ for $j=1, \ldots,n$.

\begin{lemma}\label{lem:aux}
Let $T_n(t)=\{(t_1,\ldots,t_n);0<t_1<\ldots<t_n<t\}$ for any
$t>0$ and $n \geq 1$. Then, for any $\beta_1, \ldots,\beta_n>-1$, we have:
$$I_n(t,\beta_1, \ldots,\beta_n):=\int_{T_n(t)}\prod_{j=1}^{n}(t_{j+1}-t_j)^{\beta_j}dt_1 \ldots dt_n=\frac{\prod_{j=1}^n \Gamma(\beta_j+1)}{\Gamma(|\beta|+n+1)}t^{|\beta|+n},$$
where $|\beta|=\sum_{j=1}^n \beta_j$ and we denote $t_{n+1}=t$. Consequently, if
there exist $M>\varepsilon>0$ such that $\varepsilon \leq \beta_j+1
\leq M$ for all $j=1, \ldots,n$, then
$$I_n(t,\beta_1, \ldots,\beta_n)\leq \frac{C^n}{\Gamma(|\beta|+n+1)}t^{|\beta|+n},$$
where $C=\sup_{x \in [\varepsilon,M]}\Gamma(x)$.
\end{lemma}

\medskip

At this point we proceed to state and prove the main result of the
section.

\begin{theorem}\label{thm:existence}
Equation \eqref{HAM} has a unique solution $u=\{u(t,x);t \geq 0,x\in
\bR\}$ which satisfies: for any $t \geq 0$, $x \in \bR$ and $p \geq
2$,
\begin{equation}
\label{expo-p-moment} E|u(t,x)|^p \leq |\eta|^p C_1 \exp(C_2
|\lambda|^{2/(2H+1)} p ^{(2H+2)/(2H+1)}t),
\end{equation}
where $C_1$ and $C_2$ are some positive constants which depend on
$H$.
\end{theorem}

\begin{proof} We first show the existence and uniqueness of the solution.
According to Proposition \ref{exist-sol-th}, the necessary and
sufficient condition for the existence of the solution $u$ is that
the series \eqref{series-conv} is convergent. Moreover, when it
converges, this series is equal to $E|u(t,x)|^2$.

To evaluate this series, we proceed as on page 49 of \cite{HHLNT} in
the case of the Parabolic Anderson Model with the same noise as
here, except that we have a simplified initial condition.

We fix $t>0$ and $x \in \bR$. We define the Fourier transform for
any function $\varphi \in L^1(\bR^n)$,
$$\cF \varphi(\xi_1 \ldots,\xi_n)=\int_{\bR^n}e^{-i \sum_{j=1}^n \xi_j x_j}\varphi(x_1, \ldots,x_n)dx_1 \ldots dx_n.$$

Let $f_n(\cdot,t,x)$ be the kernel given by \eqref{def-fn}. By
direct calculation, we obtain that
\begin{align*}
& \cF f_n(t_1,\cdot,\ldots,t_n,\cdot,t,x)(\xi_1, \ldots,\xi_n)\\
& \quad
 =\eta\lambda^n  e^{-i (\xi_1+\ldots,+\xi_n)x} \overline{\cF G(t_2-t_1,\cdot)(\xi_1)} \\
& \qquad \times \overline{\cF G(t_3-t_2,\cdot)(\xi_1+\xi_2)}\ldots
\overline{\cF G(t-t_n,\cdot)(\xi_1+\ldots+\xi_n)}
1_{\{0<t_1<\ldots<t_n<t\}}
\end{align*}
and hence
\begin{align*}
& \cF \widetilde{f_n}(t_1,\cdot, \ldots,t_n,\cdot,t,x)(\xi_1, \ldots,\xi_n)\\
& \quad =e^{-i (\xi_1+\ldots,+\xi_n)x} \frac{\eta\lambda^n
}{n!}\sum_{\rho \in S_n}
 \overline{\cF G(t_{\rho(2)}-t_{\rho(1)},\cdot)(\xi_{\rho(1)})} \\
& \qquad  \quad \times \overline{\cF
G(t_{\rho(3)}-t_{\rho(2)},\cdot)(\xi_{\rho(1)}+\xi_{\rho(2)})}
\ldots \overline{\cF G(t-t_{\rho(n)},\cdot)(\xi_{\rho(1)}+\ldots+\xi_{\rho(n)})}\\
& \qquad \quad  \times 1_{\{ 0 < t_{\rho(1)}< \ldots<t_{\rho(n)}<
t\}}
\end{align*}

Using \eqref{def-norm-Hn}, we obtain that:
\begin{align*}
& \|\widetilde{f}_n(\cdot,t,x)\|_{\cH^{\otimes n}}^2 \\
& \;  =\eta^2\lambda^{2n} \int_{(0,t)^n}\int_{\bR^n}|\cF \widetilde{f}_n(t_1,\cdot,\ldots,t_n,\cdot)(\xi_1, \ldots,\xi_n)|^2 \mu(d\xi_1) \ldots \mu(d\xi_n) dt_1 \ldots dt_n  \\
& \;  =\eta^2\lambda^{2n} \sum_{\rho \in
S_n}\int_{0<t_{\rho(1)}<\ldots<t_{\rho(n)}<t}\int_{\bR^n}|\cF
\widetilde{f}_n(t_1,\cdot,\ldots,t_n,\cdot)(\xi_1, \ldots,\xi_n)|^2 \\
& \hspace{8cm} \mu(d\xi_1) \ldots \mu(d\xi_n) dt_1 \ldots dt_n \\
& \; = \frac{\eta^2\lambda^{2n}}{(n!)^2}\sum_{\rho \in
S_n}\int_{0<t_{\rho(1)}<\ldots<t_{\rho(n)}<t}\int_{\bR^n}
 |\cF G(t_{\rho(2)}-t_{\rho(1)},\cdot)(\xi_{\rho(1)})|^2 \\
&  \qquad
 \times |\cF G(t_{\rho(3)}-t_{\rho(2)},\cdot)(\xi_{\rho(1)}+\xi_{\rho(2)})|^2
  \ldots |\cF G(t-t_{\rho(n)},\cdot)(\xi_{\rho(1)}+\ldots \xi_{\rho(n)})|^2 \\
  & \hspace{8cm} \mu(d\xi_{\rho(1)}) \ldots \mu(d\xi_{\rho(n)}) dt_{\rho(1)}\ldots dt_{\rho(n)} \\
 &  \; =\frac{\eta^2\lambda^{2n}}{n!} \int_{0<t_1'<\ldots<t_n'<t} \int_{\bR^n}
 |\cF G(t_2'-t_1',\cdot)(\xi_1')|^2  |\cF G(t_3'-t_2',\cdot)(\xi_1'+\xi_2')|^2 \ldots \\
  &  \qquad \times |\cF G(t-t_n',\cdot)(\xi_1'+\ldots+\xi_n')|^2 \mu(d\xi_1')\ldots \mu(d\xi_n')dt_1' \ldots dt_n',
\end{align*}
where for the last equality we used the change of variable
$\xi_{j}'=\xi_{\rho(j)}$ and $t_j'=t_{\rho(j)}$.

Recall that $T_n(t)=\{(t_1,\ldots,t_n);0<t_1<\ldots<t_n<t\}$. Hence
\begin{align}
& n! \, \|\widetilde{f}_n(\cdot,t,x)\|_{\cH^{\otimes n}}^2 \nonumber \\
& \; =\eta^2\lambda^{2n}c_H^n \int_{T_n(t)} \int_{\bR^n}
 |\cF G(t_2-t_1,\cdot)(\xi_1)|^2  |\cF G(t_3-t_2,\cdot)(\xi_1+\xi_2)|^2 \ldots \nonumber \\
 \label{norm-tilde-fn}
&  \quad \quad \quad \quad \quad |\cF G(t-t_n,\cdot)(\xi_1+\ldots+\xi_n)|^2 |\xi_1|^{1-2H}\ldots |\xi_n|^{1-2H}dt_1
\ldots dt_n \\
\nonumber &  \; =\eta^2\lambda^{2n} c_H^n \int_{T_n(t)} \int_{\bR^n}
 |\cF G(t_2-t_1,\cdot)(\eta_1)|^2  |\cF G(t_3-t_2,\cdot)(\eta_2)|^2 \ldots |\cF G(t-t_n,\cdot)(\eta_n)|^2  \\
 \nonumber
&  \quad \quad \quad \quad \quad |\eta_1|^{1-2H}
|\eta_2-\eta_1|^{1-2H}\ldots |\eta_n-\eta_{n-1}|^{1-2H} d\eta_1
\ldots d\eta_n dt_1 \ldots dt_n,
\end{align}
where for the last equality we used the change of variable
$\eta_j=\xi_1+\ldots+\xi_j$ for $j=1, \ldots,n$.

Using the inequality $(a+b)^p \leq a^b+b^p$ for $p \in (0,1)$ and
$a,b>0$, we have:
$$|\eta_j-\eta_{j-1}|^{1-2H} \leq (|\eta_{j-1}|+|\eta_{j}|)^{1-2H}\leq |\eta_{j-1}|^{1-2H}+|\eta_j|^{1-2H}.$$
We use the following fact: for any finite set $S$ and positive
numbers $(a_j)_{j \in S}$ and $(b_j)_{j \in S}$,
\begin{equation}
\label{prod-ab}
\prod_{j \in S}(a_j+b_j)=\sum_{I \subset S} \left(\prod_{j \in I} a_j \right) \left(\prod_{J \in S \setminus I}b_j \right).
\end{equation}
Hence,
\begin{align*}
\prod_{j=2}^{n}|\eta_j-\eta_{j-1}|^{1-2H} &  \leq \prod_{j=2}^{n} (|\eta_{j-1}|^{1-2H}+|\eta_j|^{1-2H})\\
& =\sum_{I \subset \{2, \ldots,n\}} \left( \prod_{j \in
I}|\eta_{j-1}|^{1-2H}\right) \left(\prod_{j \in \{2,\ldots,n\}
\setminus
I}|\eta_j|^{1-2H} \right)\\
& = \sum_{I \subset \{2, \ldots,n\}} \left( \prod_{j \in
I-1}|\eta_{j}|^{1-2H}\right) \left(\prod_{j \in
\{2,\ldots,n\}\setminus I}|\eta_j|^{1-2H} \right),
\end{align*}
where $I-1=\{j-1;j \in I\}$. Note that the last sum can be written
as
$$\sum_{\alpha \in D_n} \prod_{j=1}^{n}|\eta_j|^{\alpha_j},$$
where $D_n$ is a set of cardinality $2^{n-1}$ consisting of
multi-indices $\alpha=(\alpha_1,\ldots,\alpha_n)$ with the following
properties:
$$|\alpha|=\sum_{j=1}^{n}\alpha_j=(n-1)(1-2H),$$
$$\alpha_1 \in \{0,1-2H\}  \quad \mbox{and}\quad \alpha_j \in \{0,1-2H,2(1-2H)\} \ \mbox{for} \ j=2,\ldots,n.$$
Hence
\begin{equation}
\prod_{j=2}^{n}|\eta_j-\eta_{j-1}|^{1-2H} \leq \sum_{\alpha \in D_n}
\prod_{j=1}^{n}|\eta_j|^{\alpha_j}. \label{eq:0}
\end{equation}

Using the notation $t_{n+1}=t$, we obtain
\begin{eqnarray*}
\lefteqn{n! \, \|\widetilde{f}_n(\cdot,t,x)\|_{\cH^{\otimes n}}^2} \\
& & \leq \eta^2\lambda^{2n} c_H^n \int_{T_n(t)}
\int_{\bR^n}\prod_{j=1}^{n}|\cF G(t_{j+1}-t_j,\cdot)(\eta_j)|^2
|\eta_1|^{1-2H} \\
& & \quad \quad \quad \times \sum_{\alpha \in D_n}
\prod_{j=1}^{n}|\eta_j|^{\alpha_j}
d\eta_1 \ldots d\eta_n dt_1 \ldots dt_n \\
& & = \eta^2\lambda^{2n} c_H^n \sum_{\alpha \in D_n}
\int_{T_n(t)}\left( \int_{\bR} |\cF G(t_2-t_1,\cdot)(\eta_1)|^2
|\eta_1|^{1-2H+\alpha_1}d\eta_1 \right)\\
& & \quad \quad \quad \times  \prod_{j=2}^{n} \left( \int_{\bR}|\cF
G(t_{j+1}-t_j,\cdot)(\eta_j)|^2 |\eta_j|^{\alpha_j}d\eta_j \right)
dt_1 \ldots dt_n.
\end{eqnarray*}

Note that for any $\alpha \in (-1,1)$ and $t>0$, we have: (see
relation (3.3) in \cite{BJQ})
\begin{equation}
\label{33-BJQ} \int_{\bR}|\cF G(t,\cdot)(\xi)|^2
|\xi|^{\alpha}d\xi=C_{\alpha}t^{1-\alpha},
\end{equation}
where $C_{\alpha}>0$ is a constant which depends on $\alpha$. We use
this relation for $1-2H+\alpha_1$ and for $\alpha_j$ with
$j=2,\ldots,n$. In order to have $2(1-2H)<1$ we need to assume that
$H>1/4$.

It follows that:
$$n! \, \|\widetilde{f}_n(\cdot,t,x)\|_{\cH^{\otimes n}}^2 \leq \eta^2\lambda^{2n} C^n \sum_{\alpha \in D_n}\int_{T_n(t)} (t_2-t_1)^{2H-\alpha_1}\prod_{j=2}^{n}(t_{j+1}-t_j)^{1-\alpha_j}dt_1 \ldots dt_n,$$
where $C>0$ is constant depending on $H$.
At this point, we apply Lemma \ref{lem:aux} with
$\beta_1=2H-\alpha_1$ and $\beta_j=1-\alpha_j$ for all $j=2,
\ldots,n$. We note that $\beta_j \in [0,1]$ for all $j=1,\ldots,n$
and hence, we can take $\varepsilon=1$ and $M=2$ in Lemma \ref{lem:aux}. Then
$$|\beta|=\sum_{j=1}^{n}\beta_j=2H+(n-1)-|\alpha|=2Hn.$$

We obtain that
\begin{equation}
\label{bound-norm-fn} n! \,
\|\widetilde{f}_n(\cdot,t,x)\|_{\cH^{\otimes n}}^2 \leq \eta^2
\frac{\lambda^{2n} C^n t^{n(2H+1)}}{\Gamma(n(2H+1)+1)} \leq \eta^2
\frac{\lambda^{2n} C^n t^{n(2H+1)}}{(n!)^{2H+1}},
\end{equation}
where $C>0$ is a constant which depends on $H$. For the last
inequality we used the fact that $\Gamma(an+1)\geq C(n!)^a$ for all
$n \geq 1$ and for any $a>1$ (see e.g. relation (68) of
\cite{balan-conus}). Using Lemma A.1 of \cite{balan-conus}, we infer
that
\begin{eqnarray*}
\sum_{n \geq 0}n! \, \|\widetilde{f}_n(\cdot,t,x)\|_{\cH^{\otimes
n}}^2 \leq \eta^2 \sum_{n \geq 0} \frac{\lambda^{2n} C^n
t^{n(2H+1)}}{(n!)^{2H+1}} \leq  \eta^2 C_1
\exp(|\lambda|^{2/(2H+1)}C_2 t),
\end{eqnarray*}
for some positive constants $C_1$ and $C_2$ which depend on $H$.
This proves that the solution exists, is unique and satisfies
\eqref{expo-p-moment} for $p=2$.

To obtain \eqref{expo-p-moment} for $p\geq 2$ arbitrary, we proceed
as in the proof of Proposition 5.1 of \cite{balan-conus}. We denote
by $\|\cdot \|_p$ the $L^p(\Omega)$-norm. Using Minkowski's
inequality, the equivalence of the $\|\cdot\|_p$-norms on a {\em
fixed} Wiener chaos space $\cH_n$ and \eqref{bound-norm-fn}, we have
\begin{eqnarray*}
\|u(t,x)\|_p & \leq & \sum_{n \geq 0}\|I_n(f_n(\cdot,t,x))\|_p \leq \sum_{n \geq 0}(p-1)^{n/2}\|I_n(f_n(\cdot,t,x))\|_2 \\
& =& \sum_{n \geq 0}(p-1)^{n/2}\left(n! \|\widetilde{f}_n(\cdot,t,x))\|_{\cH^{\otimes n}}^{2} \right)^{1/2}\\
&\leq &  |\eta| \sum_{n \geq 0}(p-1)^{n/2} \frac{|\lambda|^{n}
C^{n/2}t^{n(H+1/2)}}{(n!)^{H+1/2}}.
\end{eqnarray*}
Using again Lemma A.1 of \cite{balan-conus}, we obtain:
$$\|u(t,x)\|_p \leq  |\eta| C_1 \exp(C_2 |\lambda|^{2/(2H+1)} p^{1/(2H+1)}t).$$
Relation \eqref{expo-p-moment} follows taking power $p$.
\end{proof}

\medskip

Finally, the optimality of the condition $H> \frac14$ is addressed in the following proposition.

\begin{proposition}\label{prop:optimality}
 Assume that $H\in (0,\frac12)$ and that equation \eqref{HAM} admits a unique solution in the Skorohod sense.
 Then $H>\frac14$.
\end{proposition}

\begin{proof}
 By Proposition \ref{exist-sol-th}, equation \eqref{HAM} has a solution if and only if
 condition \eqref{series-conv} holds. Note that, for any $n\geq 1$, in \eqref{norm-tilde-fn}
 we proved that
\begin{align*}
& n! \, \|\widetilde{f}_n(\cdot,t,x)\|_{\cH^{\otimes n}}^2 \nonumber \\
&  \; =\eta^2\lambda^{2n} c_H^n \int_{T_n(t)} \int_{\bR^n}
 |\cF G(t_2-t_1,\cdot)(\eta_1)|^2  |\cF G(t_3-t_2,\cdot)(\eta_2)|^2 \ldots |\cF G(t-t_n,\cdot)(\eta_n)|^2  \\
 \nonumber
&  \quad \quad \quad \quad \quad \times |\eta_1|^{1-2H}
|\eta_2-\eta_1|^{1-2H}\ldots |\eta_n-\eta_{n-1}|^{1-2H} d\eta_1
\ldots d\eta_n dt_1 \ldots dt_n,
\end{align*}
where we recall that $T_n(t)=\{(t_1,\ldots,t_n);0<t_1<\ldots<t_n<t\}$.
Assume that $n\geq 2$ and note that if $a,b\in \bR$ have opposite signs, then $|a-b|=|a|+|b|$.
Based on this simple observation, we estimate from below the integral with respect to $(\eta_1,\eta_2)\in \bR^2$
by the integral on the set $(\eta_1,\eta_2)\in \bR_+\times \bR_-$ and use that, for any
$(\eta_1,\eta_2)\in \bR_+\times \bR_-$, we have
\[
 |\eta_2-\eta_1|=|\eta_1| + |\eta_2| \geq |\eta_1|.
\]
Hence
\[
   |\eta_2-\eta_1|^{1-2H} \geq |\eta_1|^{1-2H}
\]
since the function $\xi\mapsto \xi^{1-2H}$ is increasing on $\bR$ (because $1-2H>0$).
Then, by Fubini theorem,
\begin{align*}
 n! \, \|\widetilde{f}_n(\cdot,t,x)\|_{\cH^{\otimes n}}^2
&  \geq \eta^2\lambda^{2n} c_H^n \int_{T_n(t)}
\left(\int_{\bR_+} |\cF G(t_2-t_1,\cdot)(\eta_1)|^2 |\eta_1|^{2-4H} \, d\eta_1 \right) \\
& \quad \quad \quad \quad \quad \times
 \int_{\bR_-}  \int_{\bR^{n-2}}
   |\cF G(t_3-t_2,\cdot)(\eta_2)|^2 \ldots |\cF G(t-t_n,\cdot)(\eta_n)|^2  \\
 \nonumber
&  \quad \quad \quad \quad \quad \times |\eta_3-\eta_2|^{1-2H} \cdots  |\eta_n-\eta_{n-1}|^{1-2H} d\eta_2
\ldots d\eta_n dt_1 \ldots dt_n.
\end{align*}
For all $r>0$, the function $\cF G(r,\cdot)(\eta_1)=\frac{\sin(r|\eta_1|)}{|\eta_1|}$ is symmetric,
which implies
\[
  \int_{\bR_+} |\cF G(r,\cdot)(\eta_1)|^2 |\eta_1|^{2-4H} \, d\eta_1
  = \frac 12 \int_{\bR} |\cF G(r,\cdot)(\eta_1)|^2 |\eta_1|^{2-4H} \, d\eta_1,
\]
and the latter integral is convergent if and only if $H\in (\frac14,\frac34)$ (see \cite[Eq. (3.3)]{BJQ}).
This concludes the proof.
\end{proof}


\section{Solution in the It\^o sense}
\label{sec:ito}

In this section, we introduce the concept of solution of equation \eqref{HAM} in the It\^o sense, as defined in \cite{BJQ}, and we show that this solution coincides with the solution in the Skorohod sense defined in the Section \ref{sec:solution}.

We recall from Section 2.2 of \cite{BJQ} the construction of the stochastic integral with respect to the noise $X$.
We say that $g$ is an {\em elementary process} on $\bR_+ \times \bR$ if it is of the form
\begin{equation}
\label{eq:34}
g(\omega,t,x)=Y(\omega)1_{(a,b]}(t) 1_{(v,w]}(x),
\end{equation}
where $0 \leq a<b\le T$, $Y$ is a $\bR$-valued bounded
$\cF_a$-measurable random variable, and $v,w \in \bR$ with $v<w$.
The integral of a process $g$ of the form \eqref{eq:34} with respect to $X$ is defined by:
$$(g\cdot X)_t:=\int_0^t \int_{\bR}g(s,x)X(ds,dx)=Y(X_{t \wedge b}((v,w])-X_{t \wedge a}((v,w])).$$
This definition is extended to the set ${\cal E}_r$ of all linear combinations of elementary processes. For any process $g \in {\cal E}_r$ and for any $T>0$, we have
$$E\left|\int_0^T \int_{\bR}g(t,x)X(dt,dx)\right|^2 =E\int_0^T \int_{\bR}|\cF g(t,\cdot)(\xi)|^2 \mu(d\xi)dt=:\|g\|_{0,T}^2.$$

We fix $T>0$. Then the map $g \mapsto g \cdot X$ is an isometry which is extended to the completion ${\cal P}_0^T$ of ${\cal E}_r$ with respect to the norm
$\|g\|_{0,T}^2$.
For any $g \in {\cal P}_0^T$, we say that
$$(g \cdot X)_t=\int_0^t \int_{\bR}g(s,x)X(ds,dx), \quad t \in [0,T]$$
is the {\bf It\^o integral} of $g$ with respect to $X$.

\smallskip

Following \cite{BJQ}, we have the following definition.

\begin{definition}
{\em We say that a
random field $u=\{u(t,x);\, t\in [0,T]\times \bR\}$ is a {\bf
solution of \eqref{HAM} (in the It\^o sense)} if, for any  $(t,x)\in [0,T]\times \bR$,
\begin{equation}
\label{Ito-eq} u(t,x)=\eta+\int_0^t \int_{\bR}G(t-s,x-y)\lambda
u(s,y)X(ds,dy) \quad \rm{a.s.}
\end{equation}
}
\end{definition}

By Theorem 1.1 of \cite{BJQ}, we know that \eqref{HAM} has a unique solution in the It\^o sense, which is obtained as the limit of the sequence $(u_n)_{n \geq 0}$ of Picard iterations, given by: $u_0(t,x)=\eta$ and
\begin{equation}
\label{Picard}
u_{n+1}(t,x)=\eta+\int_0^t \int_{\bR}G(t-s,x-y)\lambda u_n(s,y)X(ds,dy), \quad n \geq 0.
\end{equation}

The following result is the analogue of Proposition 1.3.11 of \cite{nualart06} for the noise $X$; see also Theorem 2.9 of \cite{DOP09}. Its proof is included in Appendix \ref{appendix}, for the sake of completeness.

\begin{theorem}
\label{thm:Skorohod=Ito} Let $u=\{u(t,x);t \ge 0,\,x \in \bR\}$ be a
 process such that $u$ restricted to $[0,T]$ belongs to $\mathcal{P}_0^T$ for any $T> 0$. Then for
any $t>0$, $u\,\mathbf{1}_{[0,t]} \in Dom\ \delta$ and its Skorohod
integral coincides with the It\^o integral, that is
$$\int_0^{\infty}\int_{\bR}u(s,x)\mathbf{1}_{[0,t]}(s) X(\delta s,\delta x)=\int_0^{t}\int_{\bR}u(s,x)X(ds,dx).$$
\end{theorem}


\medskip

The following result is an immediate consequence of Theorem \ref{thm:Skorohod=Ito}.

\begin{theorem}
The solution of equation \eqref{HAM} in the It\^o sense coincides with the solution of \eqref{HAM} in the Skorohod sense. Moreover, the
$n$-th Picard iteration is given by a predictable modification of
\begin{equation}
\label{def-un} u_n(t,x)=\sum_{k=0}^n
I_k(f_k(\cdot,t,x)),
\end{equation}
where $f_k(\cdot,t,x)$ is defined by \eqref{def-fn} for $k \geq 1$
and $f_0(t,x)=\eta$.
\end{theorem}

\begin{proof} Let $u$ be the solution of \eqref{HAM} in the It\^o sense. Fix $(t,x) \in \bR_{+} \times \bR$ and consider the process
$v^{(t,x)}(s,y)=1_{[0,t]}(s)\lambda G(t-s,x-u)u(s,y)$ for $s>0$ and $y \in \bR$.  Then $v^{(t,x)}$ restricted to
$[0,T]$ belongs to $ \mathcal P_0^T$ for any $T>0$.
By Theorem \ref{thm:Skorohod=Ito}, $v^{(t,x)} \in {\rm Dom} \ \delta$
and its Skorohod integral coincides with its It\^o integral, i.e. the Skorohod integral in \eqref{def-sol} coincides with the It\^o integral in \eqref{Ito-eq}. Hence, $u$ is the solution of \eqref{HAM} in the Skorohod sense.

For the second statement, let $(u_n)_{n \geq 0}$ be the sequence defined by \eqref{def-un}. It can be proved that each $u_n$ is $L^2(\Omega)$-continuous, and hence it has a predictable modification. We work with this modification (denoted also by $u_n$). We write \eqref{def-un} for $u_n(s,y)$, and we multiply this relation by $1_{[0,t]}(s) \lambda
G(t-s,x-y)$. We obtain:
\begin{eqnarray*}
1_{[0,t]}(s)G(t-s,x-y) u_n(s,y)&=&\sum_{k=0}^{n}I_k(1_{[0,t]}(s) \lambda G(t-s,x-y)f_k(\cdot,s,y))\\
&=& \sum_{k=0}^{n}I_{k}(f_{k+1}(\cdot,s,y,t,x)).
\end{eqnarray*}
By Proposition \ref{integr-criterion}, the process
$1_{[0,t]}(\cdot)\lambda G(t-\cdot,x-\cdot) u_n$ is Skorohod
integrable and its Skorohod integral is given by
$$\int_0^t \int_{\bR} G(t-s,x-y)\lambda u_n(s,y)X(\delta s,\delta y)=
\sum_{k=0}^{n}I_{k+1}(f_{k+1}(\cdot,t,x))=\eta-u_{n+1}(t,x).$$
Now, we use this equality to show by induction that
the sequence $(u_n)_{n \geq 0}$ satisfies exactly the recurrence
relation \eqref{Picard}. Indeed, since $u_0$ is deterministic, one clearly obtains
 that relation \eqref{Picard} holds for $n=0$. In \cite[Thm. 1.1]{BJQ}, we proved that
 any Picard iterate defined through \eqref{Picard} is well-defined. Therefore,
 the restriction of the process
$1_{[0,t]}(\cdot)\lambda G(t-\cdot,x-\cdot) u_1$
 to $[0,T]$ belongs to ${\cal P}_0^T$ for any $T>0$,
and hence, by Theorem \ref{thm:Skorohod=Ito}, its Skorohod integral
coincides with the It\^o integral. This shows
that relation \eqref{Picard} holds for $n=1$. The same argument can be used
in the general induction step.
\end{proof}


\section{Intermittency}
\label{sec:inter}

In this section, we prove that the solution $u$ to equation \eqref{HAM} is weakly intermittent.
Recall the definitions \eqref{def-Lyapunov} and \eqref{def-Lyapunov-up} of the lower and upper Lyapunov exponents
of $u$ of
order $p$, respectively. We will prove the following result.

\begin{theorem}\label{thm:main}
The random field $u$ is weakly intermittent, i.e.
  $$\underline{\gamma}(2)>0 \quad \mbox{and} \quad \overline{\gamma}(p)<\infty \quad \mbox{for all} \ p \geq 2.$$
\end{theorem}

\begin{proof}
For the upper bound, we observe that the estimate \eqref{expo-p-moment} clearly implies that
$\overline{\gamma}(p)<\infty$ for all $p \geq 2$. So it remains
to prove the lower bound, i.e. $\underline{\gamma}(2)>0$. For this,
we will obtain an exponential lower bound for the second order moment of the solution $u$.
Surprisingly, we will see that our noise's roughness (i.e. the fact that $H<\frac12$) works in our advantage.

By Proposition \ref{exist-sol-th}, we know that, for all $(t,x)\in \bR_+\times \bR$,
\[
 E|u(t,x)|^2 = \sum_{n=0}^\infty n! \, \|\widetilde{f}_n(\cdot,t,x)\|_{\cH^{\otimes n}}^2,
\]
where the kernels $\widetilde{f}_n$ are given in \eqref{def-fn-sym}. On the other hand, equation
\eqref{norm-tilde-fn} in the proof of Theorem \ref{thm:existence} yields
\begin{align*}
& n! \, \|\widetilde{f}_n(\cdot,t,x)\|_{\cH^{\otimes n}}^2 \nonumber \\
&  \; =\eta^2\lambda^{2n} c_H^n \int_{T_n(t)} \int_{\bR^n}
 |\cF G(t_2-t_1,\cdot)(\eta_1)|^2  |\cF G(t_3-t_2,\cdot)(\eta_2)|^2 \ldots |\cF G(t-t_n,\cdot)(\eta_n)|^2  \\
 \nonumber
&  \quad \quad \quad \quad \quad \times |\eta_1|^{1-2H}
|\eta_2-\eta_1|^{1-2H}\ldots |\eta_n-\eta_{n-1}|^{1-2H} d\eta_1
\ldots d\eta_n dt_1 \ldots dt_n,
\end{align*}

In order to bound the above term from below, we are going to use the same idea as in the proof of Proposition \ref{prop:optimality}.
Namely, we consider the
set $A=\{(\eta_1,\ldots,\eta_n); \eta_1 \in \bR_{+}, \eta_2 \in
\bR_{-},\ldots, \eta_{n-1}\in \bR_{+},\eta_n \in \bR_{-} \}$ if $n$
is even and $A=\{(\eta_1,\ldots,\eta_n);\eta_1\in \bR_+, \eta_2 \in
\bR_{-}, \ldots, \eta_{n-1} \in \bR_-, \eta_{n} \in \bR_{+} \}$ if
$n$ is odd. For any $(\eta_1,\ldots,\eta_n)\in A$, we have
$$|\eta_j-\eta_{j-1}|=|\eta_j|+|\eta_{j-1}| \geq |\eta_j|, \quad \mbox{for all} \ j=2, \ldots,n,$$
and hence
$$|\eta_j-\eta_{j-1}|^{1-2H} \geq |\eta_j|^{1-2H}, \quad \mbox{for all} \ j=2, \ldots,n,$$
since the function $\xi \mapsto \xi^{1-2H}$ is increasing on $\bR_+$
(because $1-2H>0$). Then
\begin{align*}
& n! \, \|\widetilde{f}_n(\cdot,t,x)\|_{\cH^{\otimes n}}^2 \nonumber \\
&  \; \geq \eta^2\lambda^{2n} c_H^n \int_{T_n(t)} \int_A
 |\cF G(t_2-t_1,\cdot)(\eta_1)|^2  |\cF G(t_3-t_2,\cdot)(\eta_2)|^2 \ldots |\cF G(t-t_n,\cdot)(\eta_n)|^2  \\
 \nonumber
&  \quad \quad \quad \quad \quad \times |\eta_1|^{1-2H}
|\eta_2|^{1-2H}\ldots |\eta_n|^{1-2H} \, d\eta_1
\ldots d\eta_n dt_1 \ldots dt_n \\
& \; = \eta^2\lambda^{2n} C^n \int_{T_n(t)}  \prod_{j=1}^n \int_{\bR} |\cF G(t_{j+1}-t_j,\cdot)(\eta)|^2
|\eta|^{1-2H} \, d\eta dt_1 \ldots dt_n,
\end{align*}
where $t_{n+1}=t$, and we have taken into account that, for all $r>0$,
\[
 \int_{\bR_+} |\cF G(r,\cdot)(\eta)|^2 |\eta|^{1-2H}  d\eta=
 \int_{\bR_-} |\cF G(r,\cdot)(\eta)|^2 |\eta|^{1-2H} d\eta
 = \frac 12 \int_{\bR} |\cF G(r,\cdot)(\eta)|^2 |\eta|^{1-2H} d\eta.
\]
Then , since $\int_{\bR} |\cF G(r,\cdot)(\eta)|^2 |\eta|^{1-2H} d\eta
 = C\, r^{2H}$ (see, e.g. \cite[Eq. (3.3)]{BJQ}),
we can infer that
\begin{align*}
 n! \, \|\widetilde{f}_n(\cdot,t,x)\|_{\cH^{\otimes n}}^2
& \geq \eta^2\lambda^{2n} C^n \int_{T_n(t)}
\prod_{j=1}^n (t_{j+1}-t_j)^{2H} \, dt_1 \ldots dt_n \\
& = \eta^2\lambda^{2n} C^n  \frac{\Gamma(2H+1)^n}{\Gamma(2Hn+n+1)}\,  t^{(2H+1)n} \\
& = \eta^2\lambda^{2n} C^n \frac{t^{(2H+1)n}}{\Gamma((2H+1)n+1)}.
\end{align*}
Here we have applied Lemma 3.3 with $\beta_j=2H$ for all $j$.

At this point, we apply the following fact, which can be verified by applying Stirling's Formula:
\[
 \frac{\Gamma(an+1)}{ a^{an} n^{\frac{1-a}{2}}  (n!)^a}\leq C, \quad \text{for all} \;\; n\geq 0\;\; \text{and}\;\;
 a\in \bR_+.
\]
We apply this inequality with $a=2H+1$. Hence
\begin{align*}
  n! \, \|\widetilde{f}_n(\cdot,t,x)\|_{\cH^{\otimes n}}^2 &
  \geq \eta^2\lambda^{2n} C^n \frac{t^{(2H+1)n}}{(n!)^{2H+1} (2H+1)^{(2H+1)n} n^{-H}} \\
  & = \eta^2\lambda^{2n} C^n \, \frac{n^H t^{(2H+1)n}}{(n!)^{2H+1}} \\
  & \geq \eta^2\lambda^{2n} C^n \, \frac{t^{(2H+1)n}}{(n!)^{2H+1}},
\end{align*}
because $n^H\geq 1$ for all $n\geq 1$.

Therefore
\begin{align*}
 \sum_{n=0}^\infty n! \, \|\widetilde{f}_n(\cdot,t,x)\|_{\cH^{\otimes n}}^2
 & = \eta^2 + \sum_{n=1}^\infty n! \, \|\widetilde{f}_n(\cdot,t,x)\|_{\cH^{\otimes n}}^2 \\
 & \geq \eta^2 \sum_{n=0}^\infty \lambda^{2n} C^n \, \frac{t^{(2H+1)n}}{(n!)^{2H+1}} \\
 & \geq C \exp(C \lambda^{\frac{2}{2H+1}}\, t),
\end{align*}
where the last inequality follows by Lemma \ref{lemmaB} below. This concludes the proof.
\end{proof}

\medskip

In the proof of the next lemma, we will apply the following inequality, which can be easily checked by induction:
for any sequence $(a_n)_{n \geq 0}$ of positive real numbers and for any $p \geq 1$, it holds
\begin{equation}
  \left(\sum^\infty_{n=0} a_n\right)^p \leq 2^{p-1}\sum^\infty_{n=0} \big(2^{p-1}\big)^n \, a_n^p.
  \label{eq:an}
\end{equation}

\smallskip

\begin{lemma}
\label{lemmaB}
For any $p>0$ and $x>0$,
$$\sum_{n\geq 0}\frac{x^n}{(n!)^p} \geq c_1 \exp(c_2 x^{1/p}),$$
where $c_1$ and $c_2$ are some positive constants which depend on $p$.
\end{lemma}

\begin{proof} {\em Case 1: $p<1$}. We use the fact that $(\sum_{n \geq 0}a_n)^p \leq \sum_{n \geq 0}a_n^p$, for any
positive real numbers $(a_n)_{n \geq 0}$. Hence,
$$\sum_{n\geq 0}\frac{x^n}{(n!)^p}=\sum_{n\geq 0}\left(\frac{x^{n/p}}{n!}\right)^{p}\geq \left(\sum_{n\geq 0}\frac{x^{n/p}}{n!}\right)^p=\exp(p\,x^{1/p}).$$

{\em Case 2: $p \geq 1$}. Using \eqref{eq:an}, we obtain:
\begin{eqnarray*}
\sum_{n\geq 0}\frac{x^n}{(n!)^p}& = &\sum_{n\geq 0}2^{(p-1)n}\left(\frac{2^{-(p-1)n/p}x^{n/p}}{n!}\right)^{p} \\
& \geq &  2^{1-p}\left( \sum_{n \geq 0} \frac{2^{-(p-1)n/p}x^{n/p}}{n!}\right)^p\\
&=& 2^{1-p} \exp(p \, 2^{-(p-1)/p}x^{1/p}).
\end{eqnarray*}
\end{proof}

\medskip

\begin{remark}\label{rmk:01}
 {\rm It is worth mentioning that the above prove also works for the case of the stochastic heat
 equation, i.e. the parabolic Anderson model, which would let us recover the same
 intermittency result of \cite{HHLNT}. The minor modifications only involve the formula
 for the Fourier transform of the underlying fundamental solution.
 Indeed, in this case we have
 \[
  G(t,x)=\frac{1}{\sqrt{2\pi t}}e^{-\frac{|x|^2}{2t}}, \; t>0, \, x\in \bR
  \qquad \text{and} \qquad \cF G(t,\cdot)(\xi)=e^{-t |\xi|^2/2}, \, \xi\in \bR.
 \]

 }
\end{remark}

\medskip


\appendix

\section{Proof of Theorem \ref{thm:Skorohod=Ito}}
\label{appendix}

The proof of this theorem  is based on two auxiliary lemmas. First
of all, we observe that for any Borel set  $A$ of $\bR_{+}$, any
Borel set $B\subset A^c$ and any $f,g\in\mathcal H$, the random
variables $I_1(f\mathbf{1}_A)$ and $I_1(g\mathbf{1}_B)$ are
independent. Indeed, the random vector $\big(I_1(f\mathbf{1}_A), \,
I_1(g\mathbf{1}_B)\big)$ is Gaussian and
\begin{align*}
E[I_1(f\mathbf{1}_A) I_1(g\mathbf{1}_B)] &
=\langle f\mathbf{1}_A,g\mathbf{1}_B\rangle_{\mathcal H} \\
& = \int_0^{\infty} \int_{\bR}
\mathcal{F}f(t,\cdot)(\xi) \mathbf{1}_A(t) \overline{\mathcal
Fg(t,\cdot)(\xi)} \mathbf{1}_B(t) \mu(d\xi) dt=0.
\end{align*}

On the other hand, for any $A\in \mathcal{B}(\bR_{+})$, we define
the $\sigma$-field
$$\mathcal F_A=\sigma\{X(1_{C} \varphi);\, C\in\mathcal
B_0\,, C\subset A,\, \varphi \in \cD(\bR)\}\vee \mathcal N,$$ where $\mathcal N$
are the null sets of $\mathfrak{F}$ and $\mathcal{B}_0$ are the
bounded Borel sets of $\bR_+$.

The following result is the analogue of Lemma 1.2.5 of \cite{nualart06} for the noise $X$.

\begin{lemma}
\label{lem:a1}
Let $F\in L^2(\Omega)$ with Wiener chaos decomposition
$F=\sum_{n=0}^\infty I_n(f_n)$ for some symmetric functions $f_n \in \cH^{\otimes n}$, and let $A\in \mathcal B(\bR_+)$.
Then
$$E[F|\mathcal F_A]=\sum_{n=0}^\infty I_n(f_n\mathbf{1}_A^{\otimes
n}).$$
\end{lemma}
\begin{proof}
Since the conditional expectation is $L^2(\Omega)$-continuous and
linear, it suffices to consider $F=I_n(f^{\otimes n})$, with $f\in
\mathcal H$ and $\|f\|_{\mathcal H}=1$. We denote by $f^{\otimes n}$ the function defined by:
$f^{\otimes n}(t_1,x_1, \ldots,t_n,x_n)=f(t_1,x_1) \ldots f(t_n,x_n)$.

We will proceed by induction
on $n$. For $n=1$, we have
$$E[I_1(f)\conda]=E[I_1(f\ia)+I_1(f\iac)\conda]=I_1(f\ia),$$
since $I_1(f \ia)$ is $\cF_A$-measurable and $I_1(f \iac)$ is independent of $\cF_A$.

Suppose that the assertion is true up to some $n \geq 2$ and let us check
its validity for $n+1$. We will use the following recurrent formula
for the multiple It\^{o}-Wiener integrals (see the proof of Proposition 1.1.4 of
\cite{nualart06}):
\begin{equation}
\label{recu}
I_{n+1}(g^{\otimes (n+1)})=I_n(g^{\otimes n})I_1(g)-n\,
\|g\|_{\mathcal H}^2\, I_{n-1}(g^{\otimes (n-1)}),
\end{equation}
for any $g\in\mathcal H$. Applying (\ref{recu}) to $g=f$, we have
\begin{align}
\nonumber
E[I_{n+1}(f^{\otimes (n+1)})\conda] & =E[I_n(f^{\otimes
n})I_1(f)\conda]-n E[I_{n-1}(f^{\otimes
(n-1)})\conda]\\
\label{step1}
& =E[I_n(f^{\otimes n})I_1(f)\conda]-n I_{n-1}(f^{\otimes
(n-1)}\ia^{\otimes(n-1)}),
\end{align}
where we have applied the induction hypothesis to $I_{n-1}$.
We write
\begin{equation}
\label{step2}
E[I_n(f^{\otimes n})I_1(f)\conda]=E[I_n(f^{\otimes n})I_1(f \ia)\conda]+E[I_n(f^{\otimes n})I_1(f \iac)\conda].
\end{equation}

Using the fact that $I_1(f\ia)$ is $\fa$-measurable and the
induction hypothesis, we have
\begin{equation}
\label{step3}
E[I_n(f^{\otimes
n})I_1(f\ia)\conda]=I_1(f\ia)I_n(f^{\otimes n}\ia^{\otimes n}),
\end{equation}
To deal with $E[I_n(f^{\otimes
n})I_1(f\iac)\conda]$, we write
$$I_n(f^{\otimes n})=I_n(f^{\otimes n}\ia^{\otimes n})+I_n(f^{\otimes n}(1-\ia^{\otimes
n})).$$
Note that
$$E[I_n(f^{\otimes
n}\ia^{\otimes n})I_1(f\iac)\conda]=I_n(f^{\otimes
n}\ia^{\otimes n})E[I_1(f \iac) \conda]=0,$$
since $I_n(f^{\otimes
n}\ia^{\otimes n})$ is $\cF_A$-measurable and $I_{1}(f \iac)$ is independent of $\cF_A$ with zero mean. Hence,
\begin{equation}
\label{step4}
E[I_n(f^{\otimes
n})I_1(f\iac)\conda]=E[I_n(f^{\otimes n}(1-\ia^{\otimes
n})) I_1(f \iac)\conda].
\end{equation}

We claim that:
\begin{equation}
\label{In-identity}
I_n(f^{\otimes n}(1-\ia^{\otimes
n}))=\sum_{k=1}^n {n\choose k} I_k(f^{\otimes k}\iac^{\otimes
k})I_{n-k}(f^{\otimes(n-k)}\ia^{\otimes(n-k)}).
\end{equation}
To see this, we note that by \eqref{prod-ab}, we have
\begin{eqnarray*}
f^{\otimes n}(t_1,x_1,\ldots,t_n,x_n)&=&\prod_{j=1}^{n}[f(t_j,x_j)\iac(t_j)+f(t_j,x_j)\ia(t_j)]\\
&=& \sum_{J \subset \{1,\ldots,n\}} \left[\prod_{j \in J} f(t_j,x_j)\iac(t_j)\right] \cdot \left[\prod_{j \in J^c}
f(t_j,x_j)\ia(t_j)\right].
\end{eqnarray*}
Subtracting from this
$(f^{\otimes n}1_{A}^{\otimes n}) (t_1,x_1,\ldots,t_n,x_n) =\prod_{j=1}^{n}f(t_j,x_j)\ia(t_j)$,
we are left with the sum over all subsets $J \subset \{1,\ldots,n\}$, except $J=\emptyset$.
Relation \eqref{In-identity} follows by integrating with respect to $X(dt_1,dx_1)\ldots X(dt_n,dx_n)$ and
denoting ${\rm card}(J)=k$.

We multiply \eqref{In-identity} by $I_1(f \iac)$ and we take the conditional expectation with respect to $\cF_A$. We obtain:
\begin{align*}
& E[I_n(f^{\otimes n}(1-\ia^{\otimes
n}))I_1(f \iac)|\cF_A] \\
& \qquad \quad =\sum_{k=1}^n {n\choose k} E[I_k(f^{\otimes k}\iac^{\otimes
k})I_{n-k}(f^{\otimes(n-k)}\ia^{\otimes(n-k)}) I_1(f \iac)|\cF_A].
\end{align*}
We evaluate separately each term in the sum on the right-hand of the previous relation.
The term corresponding to $k=n$ is
$$E[ I_n(f\iac^{\otimes
n}))I_1(f\iac)\conda]=0,$$ because $I_n(f\iac^{\otimes n})I_1(f\iac)$ is
independent of $\fa$ and its expectation is equal to $0$ (due to \eqref{orthog-Hn}, since $n\geq 2$).
The term corresponding to $k=2,\ldots, n-1$ is
$${n\choose k} E[I_k(f^{\otimes k}\iac^{\otimes
k})I_{n-k}(f^{\otimes(n-k)}\ia^{\otimes(n-k)})I_1(f\iac)\conda]=0,$$ because
$I_{n-k}(f^{\otimes(n-k)}\ia^{\otimes(n-k)})$ is $\fa$-measurable
and $I_1(f\iac)I_k(f^{\otimes k}\iac^{\otimes k})$ is independent of
$\fa$ with null expectation. The term corresponding to $k=1$ is
 $$n\,E[I_{n-1}(f^{\otimes(n-1)}\ia^{\otimes(n-1)})\big(I_1(f\iac)\big)^2\conda]=n\,\|f\iac\|_{\mathcal
 H}^2\,I_{n-1}(f^{\otimes(n-1)}\ia^{\otimes(n-1)}).$$
In summary,
\begin{equation}
\label{step5}
E[I_n(f^{\otimes n}(1-\ia^{\otimes
n}))I_1(f \iac)|\cF_A]=n\,\|f\iac\|_{\mathcal
 H}^2\,I_{n-1}(f^{\otimes(n-1)}\ia^{\otimes(n-1)}).
\end{equation}

Combining \eqref{step1}, \eqref{step2}, \eqref{step3}, \eqref{step4} and \eqref{step5}, we obtain:
\begin{align*}
 & E[I_{n+1}(f^{\otimes(n+1)})\conda] \\
 & \qquad =I_1(f\ia)I_n(f^{\otimes
 n}\ia^{\otimes n})+ n\,\|f\iac\|_{\mathcal
 H}^2\,I_{n-1}(f^{\otimes(n-1)}\ia^{\otimes(n-1)})-n I_{n-1}(f^{\otimes
 (n-1)}\ia^{\otimes (n-1)}) \\
 &  \qquad =I_1(f\ia)I_n(f^{\otimes  n}\ia^{\otimes n})- n\,\|f\ia\|_{\mathcal
 H}^2\,I_{n-1}(f^{\otimes(n-1)}\ia^{\otimes(n-1)})\\
 & \qquad =I_{n+1}(f^{\otimes(n+1)}\ia^{\otimes(n+1)}),
\end{align*}
where for the second equality we used the fact that $f\ia$ and $f\iac$ are orthogonal in $\mathcal H$ with
$\|f\ia\|_{\mathcal H}^2+\|f\iac\|_{\mathcal H}^2=\|f\|_{\mathcal H}^2=1$,
and for the last equality we used \eqref{recu} with $g=f \ia$.
This concludes the proof.
\end{proof}

\smallskip

We recall that a random variable $F\in L^2(\Omega)$ with the Wiener chaos expansion $F=\sum_{n \geq 0}I_n(f_n)$ for some symmetric functions $f_n \in \cH^{\otimes n}$ is Malliavin differentiable (i.e. $F \in \bD^{1,2}$) if and only if
\begin{equation}
\label{Mall-sum}
\sum_{n \geq 1}n n! \|f_n\|_{\cH^{\otimes n}}^2<\infty
\end{equation}
(see, e.g., Proposition 1.2.2 of \cite{nualart06}).
In this case,
\begin{equation}
\label{Mall-F}
D_{t,x}F=\sum_{n=1}^{\infty} n\,I_{n-1}(f_n(\cdot,t,x))
\end{equation}
(see Exercise 1.2.5 of \cite{nualart06} for the white noise case).
We also remind that $DF$ takes values in $\cH$, which is a space of functions.
The next result is the analogue of Proposition 1.2.8 of \cite{nualart06} for the noise $X$;
see also Proposition 3.12 of \cite{DOP09}.

\medskip

\begin{lemma}\label{lem:a2}
Let $F\in\mathbb D^{1,2}$ and $A\in \mathcal{B}(\bR_+)$. Then $E[F\conda]\in\mathbb D^{1,2}$ and
$$D(E[F\conda])=\ia E[DF\conda].$$
\end{lemma}

\begin{proof}
Let $F=\sum_{n=0}^{\infty} I_n(f_n)$ be the Wiener chaos expansion of $F$, for some symmetric functions
$f_n \in \cH^{\otimes n}$.
By Lemma \ref{lem:a1}, $E[F|\mathcal F_A]=\sum_{n=0}^\infty I_n(f_n\mathbf{1}_A^{\otimes
n})$.
The fact that $E[F\conda]\in\mathbb D^{1,2}$ follows from the criterion for Malliavin differentiability stated above,
since
$$\sum_{n=1}^{\infty} n \,n!\|f_n\ia^{\otimes n}\|^2_{\mathcal H^{\otimes
n}}\le \sum_{n=1}^{\infty} n \,n!\|f_n\|^2_{\mathcal H^{\otimes
n}}<\infty.$$
Moreover,
$$D_{t,x} (E[F|\cF_A])=\sum_{n \geq 1}n I_{n-1}(f_n(\cdot,t,x) 1_{A}^{\otimes (n-1)}1_A(t))=
1_A(t)\sum_{n \geq 1}n I_{n-1}(f_n(\cdot,t,x) 1_{A}^{\otimes (n-1)}).$$

On the other hand, using \eqref{Mall-F} and the fact that the conditional expectation is linear and continuous, we have
$$E[D_{t,x}F\conda]=\sum_{n=1}^{\infty} n\,E[I_{n-1}(f_n(\cdot,t,x))\conda]=\sum_{n=1}^{\infty} n
\,I_{n-1}(f_n (\cdot,t,x)\ia^{\otimes (n-1)}),$$
where the last equality follows form Lemma \ref{lem:a1}.
\end{proof}

\medskip

We recall the following results.

\begin{proposition}[Proposition 1.3.3 of \cite{nualart06}]
\label{intFu}
 Let $F\in\mathbb D^{1,2}$ and $u\in
\text{Dom}(\delta)$ such that  $F\,u\in L^2(\Omega;\mathcal H)$.
Then $Fu\in\text{Dom}(\delta)$ and the following equality holds
true:
$$\delta(Fu)=F\delta(u)-\langle DF,u\rangle_{\mathcal H},$$
provided that the right-hand side belongs to $L^2(\Omega)$.
\end{proposition}

\begin{proposition}[Proposition 1.3.6 of \cite{nualart06}]
\label{prop-1-3-6}
Let $u \in L^2(\Omega;\cH)$ and $(u_n)_{n \geq 1} \subset {\rm Dom}\ \delta$ such that
$E\|u_n-u\|_{\cH}^2 \to 0$. Suppose that there exists a random variable $G \in L^2(\Omega)$ such that
$$E(\delta(u_n)F) \to E(GF) \quad \mbox{for all} \ F \in \cS,$$
where $\cS$ is the class of smooth random variables of form \eqref{form-F}. Then $u \in {\rm Dom} \ \delta$ and $\delta(u)=G$.
\end{proposition}

\medskip

Now we are in position to prove Theorem \ref{thm:Skorohod=Ito}.

\smallskip

\noindent{\it Proof of  Theorem \ref{thm:Skorohod=Ito}.}
{\em Case 1.} {\em $u$ is a linear combination of elementary processes}.

By linearity, it is enough to assume that $u=g$, where $g$ is an elementary process of
form \eqref{eq:34}. The It\^{o}-type
integral of $g$ is given by
$$(g\cdot X)_t :=Y[X_{t\wedge b}\big((v,w]\big)-X_{t\wedge
a}\big((v,w]\big)].$$

Fix $t>0$. We will prove that $g\mathbf{1}_{[0,t]}\in
\text{Dom}(\delta)$ and
\begin{equation*}
\delta(g\mathbf{1}_{[0,t]})=(g\cdot X)_t.
\end{equation*}

Without loss of generality, we assume that $Y \in \mathbb D^{1,2}$.
(To see this, note that since $\bD^{1,2}$ is dense in $L^2(\Omega)$, there exists a sequence $(Y_n)_{n \geq 1} \subset \bD^{1,2}$ such that $E|Y_n-Y|^2 \to 0$. Define $g_n(t,x)=Y_n 1_{(a,b]}(t) 1_{(v,w]}(x)$. Then $g_n 1_{[0,t]} \in {\rm Dom} \ \delta $ and $\delta(g_n 1_{[0,t]})=(g_n \cdot X)_t$ for any $n \geq 1$. We apply Proposition \ref{prop-1-3-6} with $u=g 1_{[0,t]}$ and $u_n=g_n 1_{[0,t]}$.)

\smallskip

Observe that $(g\mathbf{1}_{[0,t]})(s,x)=Y\,\mathbf{1}_{(a\wedge
t,b\wedge t]}(s)\mathbf{1}_{(v,w]}(x)$.  Using Proposition \ref{intFu}
with $u=\mathbf{1}_{(a\wedge t,b\wedge t]}\mathbf{1}_{(v,w]}\in
\text{Dom}(\delta)$ and $F=Y$, we obtain that
$g\mathbf{1}_{[0,t]}\in \text{Dom}(\delta)$ and
$$\delta (g\mathbf{1}_{[0,t]})=Y\delta(\mathbf{1}_{(a\wedge
t,b\wedge t]}\mathbf{1}_{(v,w]})-\langle DY,\mathbf{1}_{(a\wedge
t,b\wedge t]}\mathbf{1}_{(v,w]}\rangle_{\mathcal H}=Y\,
I_1(\mathbf{1}_{(a\wedge t,b\wedge t]}\mathbf{1}_{(v,w]}),$$ due to
the fact that
\begin{align*}
\langle DY,\mathbf{1}_{(a\wedge t,b\wedge
t]}\mathbf{1}_{(v,w]}\rangle_{\mathcal H}&=\langle
D\big(E[Y|\mathcal F_a]\big),\mathbf{1}_{(a\wedge t,b\wedge
t]}\mathbf{1}_{(v,w]}\rangle_{\mathcal H}\\&=\langle
\mathbf{1}_{[0,a]} E[DY|\mathcal F_a],\mathbf{1}_{(a\wedge t,b\wedge
t]}\mathbf{1}_{(v,w]}\rangle_{\mathcal H}=0,
\end{align*}
where for the second equality we used Lemma \ref{lem:a2} with $A=[0,a]$ and the $\cF_a=\cF_{[0,a]}$.
The result follows using that
$$I_1(\mathbf{1}_{(a\wedge t,b\wedge t]}\mathbf{1}_{(v,w]})=X_{t\wedge b}\big((v,w]\big)-X_{t\wedge
a}\big((v,w]\big).$$

{\em Case 2. General case.} Let $t>0$ be arbitrary. Fix $T \geq t$. Since $u 1_{[0,T]} \in {\cal P}_0^T$, there exists a sequence $(u_n)_{n \geq 1}$ of simple processes defined on $[0,T] \times \bR$ such that $$E\|u_n-u \mathbf{1}_{[0,T]}\|_{\cH}^2 \to 0.$$ By {\em Case 1}, $u_n \mathbf{1}_{[0,t]} \in {\rm Dom}\ \delta$ and $\delta(u_n \mathbf{1}_{[0,t]})=(u_n \cdot X)_t$. By the definition of the It\^o integral, $(u_n \cdot X)_t \to (u \cdot X)_t$ in $L^2(\Omega)$. By applying Proposition \ref{prop-1-3-6}, we infer that $u \mathbf{1}_{[0,t]} \in {\rm Dom}\ \delta$ and $\delta(u \mathbf{1}_{[0,t]})=(u \cdot X)_t$.

\qed

\end{document}